\documentclass[final,1p,times]{elsarticle}
\usepackage{amssymb}
 \usepackage{amsthm}
\usepackage{amsmath,amssymb,amsopn,amsfonts,mathrsfs,amsbsy,amscd}
\usepackage{longtable}
\usepackage{multirow}
\usepackage[latin1]{inputenc}
\usepackage{booktabs} 
\usepackage{amsmath}

\usepackage[dvipsnames]{xcolor}

\newcommand{\ad}{\mathrm{ad}}
\newcommand{\Lf}{\mathrm{L}}
\newcommand{\Q}{\mathrm{Q}}

\newcommand{\esp}{\quad\mbox{and}\quad}
\newcommand{\R}{\mathbb{R}}
\newcommand{\K}{\mathbb{K}}
\newcommand{\ita}{\textit}
\newcommand{\we}{\wedge}

\newcommand{\Der}{\mathrm{Der}}

\newcommand{\G}{{\mathfrak{g}}}
\newcommand{\N}{{\mathfrak{n}}}
\newcommand{\h}{{\mathfrak{h}}}
\newcommand{\B}{{\mathfrak{B}}}
\newcommand{\om}{\omega}

\newcommand{\e}{\check{e}}

\newtheorem{theo}{Theorem}
\newtheorem{pr}{Proposition}
\newtheorem{Le}{Lemma}
\newtheorem{co}{Corollary}

\newtheorem{remark}{Remark}

\begin{document}
\begin{frontmatter}
\title{On solvable complete symplectic Lie algebras}
\author{M. Benyoussef,  M. W. Mansouri , and SM. Sbai}
\address{University of Ibn Tofail\\ Faculty of Sciences.
        Laboratory L.A.G.A\\ Kenitra-Morocco\\e-mail:
        meryem.benyoussef@uit.ac.ma
		\\ mansourimohammed.wadia@uit.ac.ma
				\\sbaisimo@netcourrier.com}				
\begin{abstract}
In this paper, we are interested in solvable complete Lie algebras, over the field $\K=\R$ or $\mathbb{C}$, which admit  a symplectic structure. Specifically, important classes are studied, and a description of complete Lie Algebra with the dimension of nilradical less or equal than six, which supported symplectic structure is given. 

\end{abstract}

        \end{frontmatter}
\section{Introduction}
A finite dimensional Lie algebra $\G$ is called a \ita{complete Lie algebra} if its center
$C(\G)$ is trivial and all of its derivations are inner, i.e. $\Der(\G) =\ad(\G)$. This definition of complete Lie algebra was given by N. Jacobson in 1962 \cite{J}. It is well known that semisimple Lie algebras, Borel subalgebras and parabolic subalgebras of  semisimple Lie algebras are complete Lie algebras. A general theory on complete Lie algebras has been developed by D. J. Meng and S. P. Wang in a series of papers \cite{M1},\cite{M2},\cite{M3},\cite{M4} and \cite{M-W}.

A \ita{symplectic Lie algebra} $(\G,\omega)$ is a Lie algebra with a skew-symmetric non-degenerate bilinear form 
$\omega$ such that for any $x,y,z\in\G$,
\begin{equation}\label{cocy}
\mathrm{d}\om(x, y, z):=\om([x, y], z)+\om([y, z], x)+\om([z, x],y) = 0,
\end{equation}
this is to say, $\omega$ is a non-degenerate $2$-cocycle for the scalar cohomology of $\G$. We call $\omega$ a symplectic form of $\G$.  Note that in such case, $\G$ must be of even dimension. A fundamental example of symplectic Lie algebras are the Frobenius Lie algebras,  i.e. Lie algebras admitting a non-degenerate exact 2-form. To our knowledge, the classification of symplectic Lie algebras, up to sympectomorphism, only exist for dimensions less than four  \cite{O} and  six-dimensional   nilpotent symplectic Lie algebras (see \cite{K-G-M} and  \cite{F} for a more recent list). A classication of a large subfamily of six-dimensional non-nilpotent solvable Lie algebras has been made by Stursberg \cite{S}. Some other special higher dimensional cases can be found, for example symplectic  filiform Lie algebras up to dimension ten \cite{G-J-K} and for the non-solvable case \cite{A-M}.

The study of symplectic Lie algebras is an active area of research.
The characterization problem of symplectic Lie algebras is still an
open problem, even though there are many interesting results on
obstructions on a Lie algebra to support a symplectic structure. Let
us recall the following well-known results  \cite{C}.
\begin{enumerate}
	\item A semisimple Lie algebra (in particular if $[\G, \G] = \G$) does
	not admit symplectic structures.
	\item The direct sum of semisimple and solvable Lie algebras cannot
	be symplectic.
	\item Unimodular symplectic Lie algebras are solvable.
	\item All symplectic Lie algebras of dimension four are solvable.
\end{enumerate}

In the present paper, we examine symplectic structures within the framework of complete Lie algebras.

The paper is organized as follows: In Section 2, we recapitulate specific findings concerning solvable complete Lie algebras extensively documented in existing literature. In Section 3, due to the complexity inherent in investigating solvable complete symplectic Lie algebras in their entirety, we opt to focus on a subclass characterized by maximal rank. Initially, our examination centers on the symplectic properties of complete solvable Lie algebras with commutative nilradicals. Subsequently, we delve into the study of complete solvable Lie algebras with filiform nilradicals. Finally, In Section 4, we determine the symplectic structure for complete Lie algebra, if present, for various nilradicals less or equal than six.

\textit{Notations}:  For $\{e_i\}_{1\leq i\leq n}$  a basis of $\G$,
we denote by $\{e^i\}_{1\leq i\leq n}$ the
dual basis on $\G^\ast$ and  $e^{i,j}$  the 2-form $e^i\wedge
e^j\in\wedge^2\G^*$. Set by  $\langle F \rangle:= \mathrm{span}_\K\{F\}$ the Lie
subalgebra  generated by the family $F$.

\section{Solvable complete Lie algebras}
In this section we resume certain results about   complete Lie algebras widely documented in the literature.

Let $\N$ be a nilpotent Lie algebra and $\h$ an abelian subalgebra of $\Der(\N)$. If all elements of $\h$ are semisimple linear transformations of $\N$, then $\h$ is called a torus on $\N$. Suppose $\h$ is a torus on $\N$, clearly $\N$ is decomposed into a direct sum of root spaces for $\h$:
\[\N=\oplus_{\beta\in\h^*}\N_\beta,\]
where $\h^*$ is the dual space of the vector space $\h$ and $\N_\beta= \{x\in\N: h.x = \beta(h)x,
h\in\h\}.$

Let $\h$ be a maximal torus on $\N$. One calls $\h$-msg a minimal system of generators which consists of root vectors for $\h$.
\begin{Le}\cite{Sa}
	Let $\h$ be a maximal torus on $\N$, $\{x_1,..., x_n\}$  an $\h$-msg and $\{\beta_1,...,\beta_n\}$
	the corresponding roots, then $\{\beta_1,...,\beta_n\}$ is a basis for the vector space $\h^*$.
\end{Le}
\begin{Le}\cite{Sa}
	Let $\N$ be a nilpotent Lie algebra and $\h_1$ , $\h_2$ two maximal tori on $\N$. Then
	\begin{equation*}
	\dim\h_1 =\dim\h_2 \leq\dim (\N/[\N,\N]).
	\end{equation*}	
\end{Le}

As all maximal tori on $\N$ are mutually conjugated, so the dimension of amaximal torus on $\N$ is an invariant of $\N$ called the rank of $\N$ (denoted by $rank(\N)$). A nilpotent Lie algebra is called maximal rank nilpotent Lie algebra if $rank(\N )=\dim (\N/[\N,\N]$. 
If $\h$ is a maximal torus on a nilpotent Lie algebra $\N$, define the bracket in $\h\oplus\N$ , by $[h_1 + n_1 , h_2 + n_2 ]=h_1(n_2)- h_2(n_1)+[n_1,n_2]$, where $h_i\in\h$, $n_i\in\N$, $i = 1, 2$, then $(\h\oplus\N,[.,.])$ is a solvable Lie algebra which will be denoted by $\h\ltimes\N$.
\begin{theo}\cite{M-S}
If $\G$ is a complete solvable Lie algebra, then it decomposes as $\h\ltimes\N$, where $\N$ is the nilradical and $\h$ is a subalgebra isomorphic to a maximal torus of $\N$. Moreover, $\h$ is a Cartan subalgebra of $\G$.
\end{theo}
\begin{theo}\cite{M-S}
	Let $\N$ be nilpotent Lie algebra of maximal rank and
	$\h$ be a maximal torus on $\N$. Then
	\begin{equation*}
	\G=\h\ltimes \N
	\end{equation*}
	is complete Lie algebra.
\end{theo}

\begin{theo}\cite{M-S}
	  Let  $\G_i=\h_i\ltimes\N_i$,  $i = 1, 2 $ be  two solvable complete  Lie algebras  with  nilpotent radical $\N_i,i  = 1, 2.$  Then  $\G_1$   is isomorphic  to  $\G_2$  if and only if $\N_1$  is isomorphic to $\N_2$. 
\end{theo}
\begin{remark}
	
	Recall that  a symplectic Lie algebra is \emph{reducible} if it has an isotropic ideal \cite{B-C}, in the contrary case the symplectic Lie algebra is called \emph{irreducible}. In the following proposition we show that any complete symplectic Lie algebra is reducible.
\end{remark}
\begin{pr}
	Let $\G$ a complete symplectic Lie algebra. Then $\G$ is reducible.
		
\end{pr}

\begin{proof}
	If $\G$ is a irreducible symplectic Lie algebra it can be written as: $\G = \h\ltimes \mathfrak{a} $.	Using \cite[Theorem 3.16]{B-C} and its proof, $\G$ has a basis $\mathcal{B}=\{f_1,\cdots,f_{2h},e_1^1,e_2^1,\cdots,e_1^m,e_2^m\}$ such that: 
	\[ 
	\h = <f_1,\cdots,f_{2h}> \quad and \quad \mathfrak{a} = <e_1^1,e_2^1>\oplus \cdots \oplus <e_1^m,e_2^m> , \quad m\geq 2h,
	\]
	
	with: 
	
	\[
	[f_i,e_1^k]=-\lambda_k(f_i)e^k_2, \quad and \quad  [f_i,e_2^k]=\lambda_k(f_i)e^k_1, \quad \lambda_k\in \h^*.
	\]
	In fact, if $\G$ is irreducible then $\G$ admits an exterior derivation.
	Consider the derivation $d$, defined as follow: \[ de_1^1=e_1^1, \quad  de_2^1=e^1_2.\]
	
	Suppose that  $d$ is interior, then there exists $x_0$ such that: $d=ad_{x_0}$, with \[
	x_0=\sum_{i=1}^{2h}{x_if_i}+\sum_{k=1}^{m}{x^ke_1^k}+\sum_{k=1}^{m}{y^ke_2^k}.
	\]
	So, $de_1^1=ad_{x_0}e_1^1$, then $e_1^1=[x_0,e_1^1]$, i.e $e_1^1=\sum_{i=1}^{2h}{-\lambda_1(f_i)e^1_2}$, which is impossible. 
	
\end{proof}

%%%%%%%%%%%%%%%%%%%%%%555

\section{Some solvable complete symplectic Lie algebras of maximal rank}
As the general study of the solvable complete symplectic Lie algebras is difficult, we choose to study a class of those Lie algebras, which are of maximal rank. We choose first to study the symplectic structure of the complete solvable Lie algebra with commutative nilradical, secondly we studied complete solvable Lie algebra with filiform nilradical.

\subsection{Complete solvable Lie algebra with commutative nilradical }
In what follows, we study the case where the Lie algebra $\N$ is commutative. Let $\N=\K^n$ we have $\Der(\N)=End(\K^n)$, so the maximal torus is $\h=\langle E_{1,1},\ldots, E_{n,n} \rangle$ then $\K^n$ is maximal rank and  $\G=\h\oplus \K^n$ is  complete Lie algebra with commutative nilradical  $\K^n$. Let $\{e_1,\ldots,e_n\}$ be a basis of $\K^n$, and $e_{n+i}=E_{i,i} $, the brackets of the Lie algebra $\G=\h\ltimes \K^n$ are given by
\begin{equation}\label{rn}
[e_{n+i},e_i]=e_i,\qquad 1\leq i\leq n.
\end{equation}

\begin{pr}\label{pr1}
	Any symplectic form over $\G=\h\ltimes \K^n$ is symplectomorphe to
\begin{equation}\label{Sym(r)n}
\om_{0}=\sum_{i=1}^{n}e^{i,n+i}+Z,
\end{equation}
with $Z\in\wedge^2\h^*$ an arbitrary $2$-form in $\h$.	Moreover $\om_0$ is an exact $2$-form if and only if $Z=0$.
\end{pr}
\begin{proof}
	An arbitrary $2$-form in $\G$ has the expression
\[\om=\sum_{1\leq i<j\leq n}a_{i,j}e^{i,j}+Z+\sum_{1\leq i,j\leq n}a_{i,n+j}e^{i,n+j},\]
with $Z\in\wedge^2\h^*$. On one side $\mathrm{d}\om(e_{n+i},e_i,e_j)=0$ implies that $a_{i,j}=0$, and on the other hand  $\mathrm{d}\om(e_{n+i},e_i,e_{n+j})=0$ implies that $a_{i,n+j}=0$ for $i\not=j$, the other conditions of $2$-cocycle are trivially verified. Then any $2$-cocycle in $\G$ is of the form  
\begin{equation*}
\om=\sum_{i=1}^{n}a_{i,n+i}e^{i,n+i}+Z.
\end{equation*}
Let us now show that $\om$ is non-degenerate if and only if $a_{i,n+i}\not=0$ for all  $1\leq i\leq n$, indeed we have
\begin{align*}
\om^n&=(\sum_{i=1}^{n}a_{i,n+i}e^{i,n+i}+\sum_{i<j}a_{n+i,n+j}e^{n+i,n+j})^n\\
&=(\sum_{i=1}^{n}a_{i,n+i}e^{i,n+i})^n\\
&=(a_{1,n+1}\ldots a_{n,2n})e^1\we\ldots\we e^{2n}.
\end{align*}
Now consider the Lie algebra automorphisms $T$ given by
$T(e_i)=\frac{1}{a_{i,n+i}}e_i$ ,\quad $T(e_{n+i})=e_{n+i}$, for $1\leq i\leq n$. We have $T^*(\om)=\om_0$.
\end{proof}
From Proposition \ref{pr1}, we get the following.
\begin{co}
	Symplectic complete Lie algebra $\h\ltimes \K^n$ has $\K^n$  as a Lagrangian ideal.
\end{co}

\subsection{Complete solvable Lie algebra with filiform nilradical}

\textbf{Filiform Lie algebra of rank $2$}: We have only two types of filiform Lie algebras of rank $2$: $\Lf_n$ and $\Q_n$.

\textbf{Case $\Lf_n$}: Let $\Lf_n$ be the $n$-dimensional Lie algebra defined by
\begin{equation*}
[e_1,e_i]= e_{i+1},\qquad i =2,\ldots,n-1,
\end{equation*}
where $\B=\{e_1,\ldots,e_n\}$ is a basis of $\Lf_n$. The maximal torus of $\Lf_n$ is spanned by $h_1$ and $h_2$ with
\begin{align*}
h_1(e_1)=e_1,& \quad h_1(e_i)=(i-2)e_i,\quad 2\leq i\leq n.\\
h_2(e_1)=0,&\quad h_2(e_i)=e_i,\quad 2\leq i\leq n.
\end{align*}
Then $\Lf_n$ is maximal rank and any  complete solvable Lie algebra    with  nilradical  $\Lf_n$, up to isomorphism, is $\G=\h\oplus \Lf_n$ with
\begin{align*}
[e_1,e_i]= e_{i+1},&\qquad\qquad\qquad\qquad\qquad  2\leq i\leq n-1,\\
[e_{n+1},e_1]=e_1,& \quad [e_{n+1},e_i]=(i-2)e_i,\quad 2\leq i\leq n.\\
[e_{n+2},e_i]=e_i,&\qquad\qquad\qquad\qquad\qquad 2\leq i\leq n.
\end{align*}
where $\Lf_n=\langle e_1,\cdots,e_n\rangle$ and $\h=\langle e_{n+1},e_{n+2}\rangle$.

\begin{pr}\label{pr2}
A complete solvable Lie algebra    with  nilradical  the filiform Lie algebra $\Lf_n$ is symplectic if and only if $n=4$.
	\end{pr}
\begin{proof}
	An arbitrary $2$-form of $\G$ can be written as
	\[\om_n=\sum_{1\leq i<j\leq n}a_{i,j}e^{i,j}+\sum_{1\leq i\leq n}a_{i,n+1}e^{i,n+1}+\sum_{1\leq i\leq n}a_{i,n+2}e^{i,n+2}+a_{n+1,n+2}e^{n+1,n+2}.\]
The condition $2$-cocycle implies

\[\left\{
\begin{array}{ll}
\mathrm{d}\om_{n}(e_{1},e_{n+1},e_{n+2})=-\om_{n}(e_{1},e_{n+2}), & \\
\mathrm{d} \om_{n}(e_{1},e_{j},e_{n+2}) = \omega_{n}(e_{j+1},e_{n+2})-\om_{n}(e_{j},e_{1}) &  2\leq j\leq n-1,\\
\mathrm{d} \om_{n}(e_{1},e_{j},e_{n+1}) = \omega_{n}(e_{j+1},e_{n+1})-\omega_{n}((j-2)e_{j},e_{1})+\omega_{n}(e_{1},e_{j}) & 2\leq j\leq n-1, \\
\mathrm{d} \om_{n}(e_{i},e_{n+1},e_{n+2}) = (2-i)\omega_{n}(e_{i},e_{n+2})+\omega_{n}(e_{i},e_{n+1}) &  2\leq i\leq n,\\
\mathrm{d} \om_{n}(e_{i},e_{n+1},e_{j}) = 2(2-i)\om_{n}(e_{i},e_{j}) &  2\leq i\neq j\leq n-1,\\
\mathrm{d} \om_{n}(e_{i},e_{n+2},e_{j}) = -2\omega_{n}(e_{i},e_{j})&  2\leq  i\neq j\leq n-1 ,\\
\mathrm{d} \om_{n}(e_{n},e_{n+2},e_{j}) = 2\omega_{n}(e_{j},e_{n}) & 2\leq j\leq n-1,\\
\mathrm{d} \om_{n}(e_{1},e_{n},e_{n+1}) = n\omega_{n}(e_{1},e_{n}), & \\
\mathrm{d} \om_{n}(e_{i},e_{n+1},e_{n}) = (1-n)\om_{n}(e_{i},e_{n}) & 2\leq i\leq n-1, \\
\mathrm{d} \om_{n}(e_{1},e_{i},e_{j}) = \omega_{n}(e_{i+1},e_{j})-\om_{n}(e_{j+1},e_{i}) & 2\leq  i\neq j\leq n-1.
\end{array}
\right.\]

This system is equivalent to

\[\left\{
\begin{array}{ll}

a_{1,n+2}=0, & \\
a_{j+1,n+2}=a_{1,j} & 2\leq j\leq n-1,\\
a_{j+1,n+1}=(j-1)a_{1,j} &2\leq  j\leq n-1,\\
a_{i,n+1}=(i-2)a_{i,n+2} & 2\leq  i\leq n, \\
a_{i,j}=0 & 2\leq  i\neq j\leq n-1, \\
a_{j,n}=0 & 2\leq  i\neq j\leq n-1, \\
%a_{1,n}=0, & \\
a_{i,n}=0 & 1\leq  i\leq n-1, \\
a_{i+1,j}=a_{j+1,i}, & 2\leq  i\neq j\leq n-1.\end{array}
\right.\]

The other conditions of $2$-cocycle are trivially verified.	Then any $2$-cocycle for the scalar cohomology of $\G=\h\oplus \Lf_n$ has the following expression
	\begin{equation*}
	\om_n=\sum_{i=1}^{n-2}a_{1,i+1}e^{1,i+1}-\sum_{i=1}^{n-2}a_{1,i+1}(ie^{i+2,n+1}+e^{i+2,n+2})+a_{1,n+1}e^{1,n+1}+a_{2,n+2}e^{2,n+2}+a_{n+1,n+2}e^{n+1,n+2},
	\end{equation*}

\[M(\om_n,\B)=\begin{pmatrix}
0 & -a_{1,2} &\cdots &-a_{1,n-1}&0&a_{1,n+1}&0 \\
a_{1,2} & 0 &\cdots &0&0&0&a_{2,n+2} \\
a_{1,3} & 0 &\cdots &0&0&a_{1,2}&a_{1,2} \\
a_{1,4}&0 &\cdots &0&0&2a_{1,3}&a_{1,3} \\
\vdots&\vdots &\vdots &\vdots&\vdots&\vdots&\vdots\\
a_{1,n-1} & 0 &\cdots &0&0&(n-3)a_{1,n-2}&a_{1,n-2} \\
0 & 0 &\cdots &0&0&(n-2)a_{1,n-1}&a_{1,n-1} \\
-a_{1,n+1} &0&\cdots &-(n-3)a_{1,n-2}&-(n-2)a_{1,n-1}&0&a_{n+1,n+2} \\
0&-a_{2,n+2}&-a_{1,2}  &\cdots &-a_{1,n-1}&-a_{n+1,n+2} &0
	\end{pmatrix}\]

	After a calculation of the determinant, we find for $n=4$ $\det(M(\om_4,\B))=a_{1,3}^{^{2}}\left(a_{1,2}^{2}-2a_{1,3}a_{2,6}\right)^{2}$, and for $n>4, \det(M(\om_n,\B))=0.$

\end{proof}
\textbf{Case $\Q_n$}: Let $\Q_n$, with $n=2k+1$ be the $n+1$-dimensional Lie algebra defined by
\begin{align*}
&[e_0,e_i]= e_{i+1},\qquad\qquad\qquad\qquad 1\leq i\leq n-1,\\
&[e_i,e_{n-i}]=(-1)^i e_{n}, \qquad\qquad\qquad 1\leq i \leq k,
\end{align*}

where $\{e_0,e_1\ldots,e_n\}$ is a basis of $\Q_n$. The maximal torus of $\Q_n$ is spanned by $h_1$ and $h_2$ with
\begin{align*}
h_1(e_0)=e_0,& \quad h_1(e_n)=(n-2)e_n \quad h_1(e_i)=(i-1)e_i,\;1\leq i\leq n-1 .\\
h_2(e_0)=0,&\quad h_2(e_n)=2e_n, \quad  h_2(e_i)=e_i,\quad 1\leq i\leq n-1.
\end{align*}
Then $\Q_n$ is maximal rank and any  complete solvable Lie algebra    with  nilradical  $Q_n$, up to isomorphism, is $\G=\h\oplus \Q_n$ with
\begin{align*}
&[e_0,e_i]= e_{i+1},\;\qquad\qquad\qquad\qquad\qquad\qquad\qquad\qquad\qquad\qquad  1\leq i\leq n-1,\\
&[e_i,e_{n-i}]=(-1)^i e_{n},\,\qquad\qquad\qquad\qquad\qquad\qquad\qquad\qquad\qquad 1\leq i \leq k,\\
&[e_{n+1},e_0]=e_0, \quad [e_{n+1},e_i]=(i-1)e_i, \quad [e_{n+1},e_n]=(n-2)e_n,\quad 1\leq i\leq n-1,\\
&[e_{n+2},e_i]=e_i,\quad[e_{n+2},e_n]=2e_n,\qquad\qquad\qquad\qquad\qquad\quad\qquad1\leq i\leq n-1.
\end{align*}
where $\Q_n=\langle e_0,\cdots,e_n\rangle$ and $\h=\langle e_{n+1},e_{n+2}\rangle$, note that $\dim(\G)=n+3=2(k+2)$.
\begin{pr}\label{pr}
	A complete solvable Lie algebra    with  nilradical  the filiform Lie algebra $\Q_n$ is symplectic. Morever an exat 2-form is given by
	$\om_{n}=\mathrm{d}e^0+\mathrm{d}e^n.$
\end{pr}
\begin{proof}
It suffices to show that the exact $2$-form $\om_n$ is non-degenerate.	Let $\{e^i\}_{0\leq i\leq  n+2}$ be the dual basis of $\{e_i\}_{0\leq i\leq  n+2}$. By the Maurer-Cartan equation we have:
	\begin{equation*}
	\mathrm{d}e^0= e^{0,n+1}\esp\mathrm{d}e^n=\sum_{i=1}^k(-1)^ie^{n-i,i}+e^{n-1,0}+(n-2)e^{n,n+1}+2e^{n,n+2}.
	\end{equation*}
Let $A_n$ be the 2-form in $(\we^2\G)^*$ given by
\[A_n=e^{0,n+1}+e^{n-1,0}+(n-2)e^{n,n+1}+2e^{n,n+2},\]
we have, $A_n^2=e^{0,n}\we(-4e^{n+1,n+2}+2(n-1)e^{n-1,n+1}+4e^{n-1,n+2})$ and $A_n^i=0$ for $i>2$.

Let $B_n$ be the 2-form in $(\we^2\G)^*$ given by
\[B_n=\sum_{i=1}^k(-1)^ie^{n-i,i},\]
 we have, $ B_n^k=e^1\we\ldots\we e^{n-1}$ and  $B_n^{k+1}=B_n^{k+2}=0$.
 
 Then $\om_n$ can be written as $\om_n=A_n+B_n$, and it follows that
\begin{align*}
\om_n^{k+2}=&(A_n+B_n)^{k+2}\\
=&A_n^2\we B_n^{k}\\
=&4e^0\we e^1\we\ldots\we e^{n+2},
\end{align*}
which proves that $\om_n$ is non-degenerate.
\end{proof}

\section{Symplectic Complete Solvable Lie Algebra  with nilradical $\leq 6$}

\upshape In this section we calculate the symplectic structure if it exists for the different nilradicals $\leq 6$ given by the tables in \cite{P-Winter}.  We begin with an example, then we give the tables emphasizing the nilradical, the brackets of the nilradical with the torus, the symplectic structure if it exists and finally the maximality of the rank.

We will give the proof of the case $\N=\N_{4,1}$, since all cases should be handled in a similar way. 

Let $\N=\N_{4,1}=<e_1,e_2,e_3,e_4>$, we have: $[e_2,e_4]=e_1,[e_3,e_4]=e_2$.\\
Firstly, we must calculate the derivations $Der(\N)$. The basis $\left\{D_1,...,D_7\right\}$ is given by:

\[\left\{
\begin{array}{lcl}
D_1(e_1)&=&e_1,D_1(e_3)=-e_3,D_1(e_4)=e_4,\\
D_2(e_2)&=&e_2,D_2(e_3)=2e_3,D_2(e_4)=-e_4,\\
D_3(e_2)&=&e_1,D_3(e_3)=e_2,\\
D_4(e_3)&=&e_1,D_5(e_4)=e_1,\\
D_6(e_4)&=&e_2,D_7(e_4)=e_3.\\

\end{array}
\right.\]

We set $\h=<e_5,e_6>$ such that: $e_5=D_1$ and $e_6=D_2$.

Now, we calculate the Lie brackets of $\G=\h \ltimes\N  $, which are:

$
[e_2,e_4]=e_1,[e_3,e_4]=e_2,[e_5,e_1]=e_1,[e_5,e_3]=-e_3,
[e_5,e_4]=e_4,[e_6,e_2]=e_2,[e_6,e_3]=2e_3,[e_6,e_4]=-e_4.
$

An arbitrary $2-$form of $\G$ can be written as:
\[\om=\sum_{1\leq i<j\leq 6}a_{i,j}e^{i,j}.\]

\[\left\{
\begin{array}{lll}
a_{1,2}=0,&a_{1,3}=0,&a_{1,4}=0,\\
a_{1,6}=0,&a_{2,3}=0,&a_{2,5}=0,\\
a_{1,5}=-a_{2,4},& a_{2,6}=-a_{3,4},\\
a_{3,6}=-2a_{3,5},& a_{4,6}=-a_{4,5}.
\end{array}
\right.\]

Hence: $\om = -a_{2,4}(e^{1,5}-e^{2,4})-a_{3,4}(e^{2,6}-e^{3,4})-a_{4,6}(e^{4,5}-e^{4,6})+a_{3,5}(e^{3,5}-2e^{3,6})+a_{5,6}e^{5,6}.$

In the following, the first column identifies the nilradical, the second column provides the brackets of the torus and nilradical, the third column specifies the symplectic structure if it exists, and finally, in the forth column we verify whether if the brackets are of maximal rank or not.

\subsection{Nilradical of dimension $3$}
We summarize the calculation of the dimension $3$ in this table:
\[
\begin{array}{|c|l|c|c|} \hline
 \mbox{Nilradical}&  \mbox{Brackets of } \h \ltimes\N   & \mbox{Symplectic}& \mbox{Maximal}\\
&   & \mbox{structure} & \mbox{rank}\\\hline

\N_{3,1}&[e_1,e_2]=e_3,[e_4,e_1]=e_1,[e_4,e_2]=e_2,&  \mbox{Dimension not even}&Yes\\

&[e_5,e_2]=e_2,[e_5,e_3]=e_3&&\\\hline
\end{array}
  \]

\subsection{ Nilradical of dimension $4$}

We summarize the calculation of the dimension $4$ in this table:
\[
\begin{array}{|c|l|c|c|} \hline
 \mbox{Nilradical}&  \mbox{Brackets of } \h \ltimes\N  & \mbox{Symplectic }& \mbox{Maximal }\\
&   & \mbox{structure} & \mbox{rank}\\\hline
\N_{4,1}&[e_2,e_4]=e_1,[e_3,e_4]=e_2& a_{2,4}(e^{1,5}-e^{2,4})+a_{3,4}(e^{2,6}-e^{3,4})& Yes\\

&[e_5,e_1]=e_1,[e_5,e_3]=-e_3,&-a_{3,5}(e^{3,5}-2e^{3,6})+a_{4,6}(e^{4,5}-e^{4,6})&\\
&[e_5,e_4]=e_4,[e_6,e_2]=e_2,&-a_{5,6}e^{5,6}&\\
&[e_6,e_3]=2e_3,[e_6,e_4]=-e_4.&Conditions: a_{2,4}\neq 0, 2a_{3,5}a_{2,4}\neq a_{3,4}^{2}&\\\hline

\end{array}			
	\]

\subsection{ Nilradical of dimension $5$}

%%%%%%%%%%%%%%%%%%%%%%%%%%%%%%%%%%%%%%%%%%%%%%%%%%%%%%%%%%%%%%%%%%%%%%%%%%%%%%%%%%%%%%%%%%%%%%%%%%%%%%%%%%%%%%%%%%%%%%%%%%%%%%%%%%%%%%%%%%%%%%%%%%%%%%%%%%%%%%%%%%%%%%%%%%%%%%%%%%%%%%%%%%%%%%%%%%%%%%%%%%%%%%%%%%%%%%%%%%%%%%%%%%%%%%%%%%%%%%%%%%%%%%%%%%%%%%%%%%%%%%%%%%%%%%%%%%%%%%%%%%%%%%%%%%%%%%%%%%%%%%%%%%%%%%%%%%%%%%%%%%%%%%%%%%%%%%%%%%%%%%%%%%%%%%%%%%%%%%%%%%%%%%%%%%%%%%%%%%%%%%%%%%%%%%%%%%%%%%%%%%%%%%%%%%%%%%%%%%%%%%%%%%%%%%%%%%%%%%%%%%%%%%%%%%%%%%%%%%%%%%%%%%%%%%%%%%%%%%%%%%%%%%%%%%%%%%%%%%%%%%%%%%%%%%%%%%%%%%%%%%%%%%%%%%%%%%%%%%%%%%%%%%%%%%%%%%%%%%%%%%%%%%%%%%%%%%%%%%%%%%%%%%%%%%%%%%%%%%%%%%%%%%%%%%%%%%%%%%%%%%%%%%%%%%%%%%%%%%%%%%%%%%%%%%%%%%%%%%%%%%%%%%%%%%%%%%%%%%%%%%%%%%%

We summarize the calculation of the dimension $5$ in this table:\\
\[
\begin{array}{|c|l|c|c|} \hline
 \mbox{Nilradical}&  \mbox{Brackets of } \h \ltimes\N  & \mbox{Symplectic} & \mbox{Maximal}\\
&   & \mbox{structure} & \mbox{rank}\\\hline

\N_{5,1}&[e_3,e_5]=e_1,[e_4,e_5]=e_2,	& a_{3,5}(e^{1,6}-e^{3,5})+a_{4,5}(e^{2,7}-e^{4,5})& Yes\\

&[e_6,e_1]=e_1,[e_6,e_4]=-e_4,&-a{3,8}e^{3,8}+a_{4,8}(e{4,6}-e{4,7}-e{4,8})&\\
&[e_6,e_5]=e_5,[e_7,e_2]=e_2,&+a_{5,8}(e^{5,6}-e^{5,8})-a_{6,7}e^{6,7}-a_{6,8}e^{6,8}&\\
&[e_7,e_4]=e_4,[e_8,e_3]=e_3,&-a_{7,8}e^{7,8}&\\
&[e_8,e_4]=e_4,[e_8,e_5]=-e_5&Conditions: a_{4,5}\neq0, a_{3,5}\neq 0, &\\

&&a_{3,5}a_{4,8}\neq a_{3,8}a_{4,5}&\\\hline
%%%%%%%%%%%%%%%%%%%%%%%%%%%%%%%%%%%%%%%%%%%%%%%%%%%%%%%%%%%%%%%%%%%%%%%%%%%%%%%%%%%%%%%%%%%%%%%%%%%%%%%%
\N_{5,2}&[e_2,e_5]=e_1,[e_3,e_5]=e_2, &\mbox{Dimension not even}& Yes\\
&[e_4,e_5]=e_3,[e_6,e_1]=e_1 &&\\
&[e_6,e_3]=-3e_3,[e_6,e_4]=-2e_4&&\\
&[e_6,e_5]=e_5,[e_7,e_2]=e_2&&\\
&[e_7,e_3]=2e_3,[e_7,e_4]=3e_4&&\\
&[e_7,e_5]=-e_5&&\\\hline

%%%%%%%%%%%%%%%%%%%%%%%%%%%%%%%%%%%%%%%%%%%%%%%%%%%%%%%%%%%%%%%%%%%%%%%%%%%%%%%%%%%%%%%%%%%%%%%%%%%%%%%%

\N_{5,3}&[e_3,e_4]=e_2,[e_3,e_5]=e_1, &\mbox{Dimension not even}& Yes\\

&[e_4,e_5]=e_3,[e_6,e_1]=e_1 &&\\

&[e_6,e_3]=\frac{1}{3}e_3,[e_6,e_4]=-\frac{1}{3}e_4&&\\
&[e_6,e_5]=\frac{2}{3}e_5,[e_7,e_2]=e_2&&\\
&[e_7,e_3]=\frac{1}{3}e_3,[e_7,e_4]=\frac{2}{3}e_4&&\\
&[e_7,e_5]=-\frac{1}{3}e_5&&\\\hline

%\end{array}
%%%%%%%%%%%%%%%%%%%%%%%%%%%%%%%%%%%%%%%%%%%%%%%%%%%%%%%%%%%%%%%%%%%%%%%%%%%%%%%%%%%%%%%%%%%%%%%%%%%%%%%%

%begin{array}{|c|l|c|c|} \hline	
%%%%%%%%%%%%%%%%%%%%%%%%%%%%%%%%%%%%%%%%%%%%%%%%%%%%%%%%%%%%%%%%%%%%%%%%%%%%%%%%%%%%%%%%%%%%%%%%%%%%%%%%
\N_{5,4}& [e_2, e_4] = e_1, [e_3, e_5] = e_1&a_{3,5}(e^{1,6}-e^{2,4}-e^{3,5})-a_{2,8}e^{2,8}&No\\

&[e_6, e_1] = e_1, [e_6, e_4] = e_4  &-a_{3,7}e^{3,7}+a_{4,8}(e^{4,6}-e^{4,8})+&\\

&[e_6, e_5] = e_5,[e_7,e_3]=e_3,&+a_{5,7}(e^{5,6}-e^{5,7})-a_{6,7}e^{6,7}-a_{6,8}e^{6,8}&\\
&[e_7, e_5] = -e_5,[e_8,e_2]=e_2,&-a_{7,8}e^{7,8}&\\
&[e_8,e_4]=-e_4,&Conditions: a_{7, 8}\neq 0, a_{3,5}\neq 0&\\\hline

%%%%%%%%%%%%%%%%%%%%%%%%%%%%%%%%%%%%%%%%%%%%%%%%%%%%%%%%%%%%%%%%%%%%%%%%%%%%%%%%%%%%%%%%%%%%%%%%%%%%%%%%
\N_{5,5}& [e_2, e_5] = e_1, [e_3, e_4] = e_1, &\mbox{Dimension not even}&No\\

&[e_3, e_5] = e_2,[e_6,e_1]=e_1 &&	\textcolor[rgb]{1,1,1}{Maximal}\\

&[e_6, e_3] = -e_3,[e_6,e_4]=2e_4,&&\\
&[e_6, e_5] = e_5,[e_7,e_2]=e_2,&&\\
&[e_7,e_3]=2e_3,[e_7,e_4]=-2e_4&&\\

&[e_7,e_5]=-e_5&&\\\hline

%%%%%%%%%%%%%%%%%%%%%%%%%%%%%%%%%%%%%%%%%%%%%%%%%%%%%%%%%%%%%%%%%%%%%%%%%%%%%%%%%%%%%%%%%%%%%%%%%%%%%%%%
\N_{5,6}& [e_2,e_5]=e_1,[e_3,e_4]=e_1, &-a_{3,4}(-e^{1,6}+e^{2,5}+e^{3,4})&No\\

&[e_3, e_5] = e_2, [e_4, e_5] = e_3 &+a_{2,6}(-e^{2,6}+\frac{5}{4}e^{3,5})&\\

&[e_6, e_1] = e_1,[e_6,e_2]=\frac{4}{5}.e_2,&+a_{3,6}(-e^{3,6}+\frac{5}{3}e^{4,5})&\\
&[e_6, e_3] = \frac{3}{5}e_3,[e_6,e_4]=\frac{2}{5}e_4,&-a_{4,6}e^{4,6}-a_{5,6}e^{5,6}&\\
&[e_6,e_5]=\frac{1}{5}e_5&Conditions: a_{3,4} \neq 0&\\\hline

\end{array}
\]

\subsection{ nilradical of dimension $6$}
We summarize the calculation of the dimension $6$ in this table:
\[
\begin{array}{|c|l|c|c|} \hline
 \mbox{Nilradical}&  \mbox{Brackets of } \h \ltimes\N & \mbox{Symplectic}& \mbox{Maximal}\\
&   & \mbox{structure} & \mbox{rank}\\\hline

\N_{6,1}&[e_1,e_2]=e_3,[e_1,e_3]=e_4,	& \mbox{Dimension not even}& Yes\\

&[e_1,e_5]=e_6,[e_7,e_1]=e_1,&&\\
&[e_7,e_3]=e_3,[e_7,e_4]=2e_4,&&\\
&[e_7,e_6]=e_6,[e_8,e_2]=e_2,&&\\
&[e_8,e_3]=e_3,[e_8,e_4]=e_4,&&\\

&[e_9,e_5]=e_5,[e_9,e_6]=e_6,&&\\\hline
%%%%%%%%%%%%%%%%%%%%%%%%%%%%%%%%%%%%%%%%%%%%%%%%%%%%%%%%%%%%%%%%%%%%%%%%%%%%%%%%%%%%%%%%%%%%%%%%%%%%%%%%
\N_{6,2}&[e_1,e_2]=e_3,[e_1,e_3]=e_4, &\mbox{Never admits a symplectic}&Yes\\
&[e_1, e_4] = e_5, [e_1, e_5] = e_6 &\mbox{structure}&\\
&[e_7,e_1]=e_1,[e_7,e_3]=e_3,&&	\\
&[e_7,e_4]=2e_4,[e_7,e_5]=3e_5,&&\\
&[e_7,e_6]=4e_6,[e_8,e_2]=e_2,&&\\
&[e_8,e_3]=e_3,[e_8,e_4]=e_4,&& \\
&[e_8,e_5]=e_5,[e_8,e_6]=e_6&& \\\hline

%%%%%%%%%%%%%%%%%%%%%%%%%%%%%%%%%%%%%%%%%%%%%%%%%%%%%%%%%%%%%%%%%%%%%%%%%%%%%%%%%%%%%%%%%%%%%%%%%%%%%%%%

\N_{6,3}&[e_1,e_2]=e_6,[e_1,e_3]=e_4, &\mbox{Dimension not even}&Yes\\

&[e_2, e_3]=e_5,[e_7,e_1]=e_1,&&\\

&[e_7,e_4]=e_4,[e_7,e_6]=e_6,&&\\
&[e_8,e_2]=e_2,[e_8,e_5]=e_5,&&\\
&[e_8,e_6]=e_6,[e_9,e_3]=e_3,&&\\
&[e_9,e_4]=e_4,[e_9,e_5]=e_5,&&\\\hline

%%%%%%%%%%%%%%%%%%%%%%%%%%%%%%%%%%%%%%%%%%%%%%%%%%%%%%%%%%%%%%%%%%%%%%%%%%%%%%%%%%%%%%%%%%%%%%%%%%%%%%%%

%\end{array}

%\begin{array}{|c|l|c|c|} \hline

\N_{6,4}&[e_1,e_2]=e_5,[e_1,e_3]=e_6, &\mbox{Dimension not even}&No\\

&[e_2, e_4] = e_6,[e_7,e_1]=e_1,&&\textcolor[rgb]{1,1,1}{Maximal}\\

&[e_7,e_4]=e_4,[e_7,e_5]=e_5,&&\\
&[e_7,e_6]=e_6,[e_8,e_2]=e_2,&&\\
&[e_8,e_4]=-e_4,[e_8,e_5]=e_5,&&\\

&[e_9,e_3]=e_3,[e_9,e_4]=e_4,&&\\

&[e_9,e_6]=e_6&&\\\hline

%%%%%%%%%%%%%%%%%%%%%%%%%%%%%%%%%%%%%%%%%%%%%%%%%%%%%%%%%%%%%%%%%%%%%%%%%%%%%%%%%%%%%%%%%%%%%%%%%%%%%%%%
\N_{6,5}& [e_1,e_3]=e_5,[e_1,e_4]=e_6, &a_{6,8}(ae^{1,3}+ae^{2,4}-ae^{5,7}-ae^{5,9}&Yes\\

&[e_2, e_3] = a*e_6, [e_2, e_4] = e_5&-e^{6,8}-ae^{6,10})+a_{6,9}(e^{1,4}+ae^{2,3}&\\

&[e_7,e_1]=e_1,[e_7,e_2]=e_2,&-e^{5,8}-ae^{5,10}-e^{6,7}-e^{6,8})-a_{2,8}(e^{1,7}+e^{2,8})&\\
&[e_7,e_5]=e_5,[e_7,e_6]=e_6,&-a_{1,8}(e^{1,8}+ae^{2,7})-a_{4,10}(e^{3,9}+ae^{4,10})&\\
&[e_8,e_1]=\frac{1}{a}.e_2,[e_8,e_2]=e_1,&-a_{4,9}(e^{4,9}+ae^{3,10})-a_{7,8}e^{7,8}-a_{7,9}e^{7,9}&\\
&[e_8,e_5]=e_6,[e_8,e_6]=\frac{1}{a}e_5,&-a_{7,10}e^{7,10}-a_{8,9}e^{8,9}-a_{8,10}e^{8,10}-a_{9,10}e^{9,10}&\\
&[e_9,e_3]=e_3,[e_9,e_4]=e_4,&&\\
&[e_9,e_5]=e_5,[e_9,e_6]=e_6,&Conditions: a\neq 0,  a_{6,8}^{2}\neq a_{6,9}^{2},&\\\
&[e_{10},e_3]=ae_4,[e_{10},e_4]=e_3&a a_{7,8}+a a_{8,9}\neq a_{7,10}+a_{9,10} & \\
&[e_{10},e_5]=ae_6,[e_{10},e_6]=e_5&&\\\hline

%%%%%%%%%%%%%%%%%%%%%%%%%%%%%%%%%%%%%%%%%%%%%%%%%%%%%%%%%%%%%%%%%%%%%%%%%%%%%%%%%%%%%%%%%%%%%%%%%%%%%%%%
\N_{6,6}& [e_1,e2]=e_6,[e_1,e_3]=e_4, &a_{6,7}(\frac{1}{3}.e^{1,2}-e^{6,7})+a_{4,8}(e^{1,3}&No\\
& [e_1, e_4] = e_5, [e_2, e_3] = e_5&-e^{4,7}-e^{4,8})+a_{5,8}(e^{1,4}+e^{2,3}&\\
&[e_7,e_1]=e_1,[e_7,e_2]=2e_2&-2e^{5,7}-e^{5,8})-a_{1,7}e^{1,7}-a_{2,7}e^{2,7}&\\
&[e_7,e_4]=e_4,[e_7,e_5]=2e_5&-a_{3,8}e^{3,8}-a_{7,8}e^{7,8}&\\
&[e_7,e_6]=3e_6,[e_8,e_3]=e_3&Conditions: a_{6,7}\neq 0, a_{5,8}\neq 0&\\
&[e_8,e_4]=e_4,[e_8,e_5]=e_5&&\\\hline

%%%%%%%%%%%%%%%%%%%%%%%%%%%%%%%%%%%%%%%%%%%%%%%%%%%%%%%%%%%%%%%%%%%%%%%%%%%%%%%%%%%%%%%%%%%%%%%%%%%%%%%%

\end{array}\]
\[
\begin{array}{|c|l|c|c|} \hline
\N_{6,7}&[e_1,e_3]=e_4,[e_1,e_4]=e_5,  &\mbox{Dimension not even}&Yes\\

& [e_2, e_3] = e_6,[e_7,e_1]=e_1&&\\
&[e_7,e_4]=e_4,[e_7,e_5]=2e_5&&\\
&[e_8,e_2]=e_2,[e_8,e_6]=e_6&&\\
&[e_9,e_3]=e_3,[e_9,e_4]=e_4&&\\
&[e_9,e_5]=e_5,[e_9,e_6]=e_6&&\\\hline

%%%%%%%%%%%%%%%%%%%%%%%%%%%%%%%%%%%%%%%%%%%%%%%%%%%%%%%%%%%%%%%%%%%%%%%%%%%%%%%%%%%%%%%%%%%%%%%%%%%%%%%%

\N_{6,8}& [e_1,e_2]=e_3+e_5, &-a_{1,2}e^{1,2}+a_{4,8}(e^{1,3}-2e^{4,7}-e^{4,8})&No\\

&[e_1, e_3] = e_4, [e_2, e_5] = e_6 &-a_{1,7}e^{1,7}+a_{6,8}(\frac{1}{2}e^{2,5}-\frac{1}{2}e^{6,7}-e^{6,8})&\\
&[e_7,e_1]=e_1,[e_7,e_3]=e_3&-a{2,8}e^{2,8}+(a_{1,2}+a_{5,8})(e^{3,7}+e^{3,8})&\\
&[e_7,e_4]=2e_4,[e_7,e_5]=e_5&-a_{5,8}(e^{5,7}+e^{5,8})+a_{7,8}e_{7,8}&\\
&[e_7,e_6]=e_6,[e_8,e_2]=e_2&Conditions:a_{6,8}\neq 0,a_{4,8}\neq 0& \\
&[e_8,e_3]=e_3,[e_8,e_4]=e_4&&\\
&[e_8,e_5]=e_5,[e_8,e_6]=2e_6&&\\\hline

%%%%%%%%%%%%%%%%%%%%%%%%%%%%%%%%%%%%%%%%%%%%%%%%%%%%%%%%%%%%%%%%%%%%%%%%%%%%%%%%%%%%%%%%%%%%%%%%%%%%%%%%
\N_{6,9}& [e_1,e_2]=e_3,[e_1,e_3]=e_4, &a_{3,8}(e^{1,2}-e^{3,7}-e^{3,8})-a_{2,8}e^{2,8}&No\\

&[e_1, e_5] = e_6, [e_2, e_3] = e_6 &+a_{4,8}(e^{1,3}-2e^{4,7}+e^{4,8})-a_{1,7}e^{1,7}&\\

&[e_7,e_1]=e_1,[e_7,e_3]=e_3,&+\frac{1}{2}a_{6,8}(e^{1,5}+e^{2,3}-e^{6,7}-e^{6,8})+&\\
&[e_7,e_4]=2e_4,[e_7,e_6]=e_6,&-a_{5,8}e^{5,8}-a_{7,8}e^{7,8}&\\
&[e_8,e_2]=e_2,[e_8,e_3]=e_3,&Conditions: a_{6,8}\neq 0,a_{4,8}\neq 0&\\
&[e_8,e_4]=e_4,[e_8,e_5]=2e_5,&&\\
&[e_8,e_6]=2e_6&&\\\hline

%%%%%%%%%%%%%%%%%%%%%%%%%%%%%%%%%%%%%%%%%%%%%%%%%%%%%%%%%%%%%%%%%%

\N_{6,10}& [e_1,e_2]=e_3,[e_1,e_3]=e_5, &-a_{1,2}(e^{1,2}-2e^{3,7})-a_{1,8}(e^{1,8}*ae^{2,7})&No\\

&[e_1,e_4]=e_6,[e_2,e_3]=a.e_6, &a_{6,8}(ae^{1,3}+ae^{2,4}-3e^{5,7}-e^{6,8})&\\

&[e_2, e_4] = e_5,&+a_{5,8}(e^{1,4}+ae^{2,3}-e^{5,8}-3e^{6,7})&\\
&[e_7,e_1]=e_1,[e_7,e_2]=e_2,&-a_{2,8}(e^{1,7}+e^{2,8})-a_{5,8}(3e^{6,7}+e^{5,8})&\\
&[e_7,e_3]=2e_2,[e_7,e_4]=2e_4,&-a_{4,7}e^{4,7}-a_{7,8}e^{7,8}&\\
&[e_7,e_5]=3e_5,[e_7,e_6]=3e_6,&Conditions: a\neq 0, a a_{6,8}^{2}\neq a_{5,8}^{2}&\\\
&[e_8,e_1]=\frac{1}{a}e_2,[e_8,e_2]=e_1,&&\\
&[e_8,e_5]=e_6,[e_8,e_6]=\frac{1}{a}e_5,&&\\\hline

%%%%%%%%%%%%%%%%%%%%%%%%%%%%%%%%%%%%%%%%%%%%%%%%%%%%%%%%%%%%%%%%%%%%%%%%%%%%%%%%%%%%%%%%%%%%%%%%%%%%%%%%
\N_{6,11}& [e_1,e_2]=e_3,[e_1,e_3]=e_4, &a_{3,8}(e^{1,2}-e^{3,7}-e^{3,8})-a_{1,7}e^{1,7}&Yes\\

&[e_1, e_4] = e_5, [e_2, e_3] = e_6, &									+a_{4,8}(e^{1,3}-2e^{4,7}-e^{4,8})-a_{2,8}e^{2,8}&\\

&[e_7,e_1]=e_1,[e_7,e_3]=e_3,&																			+a_{5,8}(e^{1,4}-3e^{5,7}-e^{5,8})-a_{7,8}e^{7,8}&\\
&[e_7,e_4]=2e_4,[e_7,e_5]=3e_5,&																		+a_{6,8}(\frac{1}{2}.e^{2,3}-\frac{1}{2}.e^{6,7}-e^{6,8})&\\
&[e_7,e_6]=e_6,[e_8,e_2]=e_2,&Conditions:a_{5,8}\neq 0  , a_{6,8}\neq 0& \\
&[e_8,e_3]=e_3,[e_8,e_4]=e_4,&&\\
&[e_8,e_5]=e_5,[e_8,e_6]=2e_6,&&\\\hline

%%%%%%%%%%%%%%%%%%%%%%%%%%%%%%%%%%%%%%%%%%%%%%%%%%%%%%%%%%%%%%%%%%%%%%%%%%%%%%%%%%%%%%%%%%%%%%%%%%%%%%%%

\N_{6,12}& [e_1,e_3]=e_4,[e_1,e_4]=e_6, &\mbox{Dimension not even}&No\\

{\color{white} Nilradical }& [e_2, e_5] = e_6,&&\\

&[e_7,e_1]=e_1,[e_7,e_4]=e_4,&&\\
&[e_7,e_5]=2e_5,[e_7,e_6]=2e_6,&&{\color{white} Maximal }\\
&[e_8,e_2]=e_2,[e_8,e_5]=-e_5,&&\\
&[e_9,e_3]=e_3,[e_9,e_4]=e_4,&&\\
&[e_9,e_5]=e_5,[e_9,e_6]=e_6&&\\
\hline

%%%%%%%%%%%%%%%%%%%%%%%%%%%%%%%%%%%%%%%%%%%%%%%%%%%%%%%%%%%%%%%%%%%%%%%%%%%%%%%%%%%%%%%%%%%%%%%%%%%%%%%%

\end{array}
\]
\[
\begin{array}{|c|l|c|c|} \hline 
\N_{6,13}&[e_1,e_2]=e_5,[e_1,e_3]=e_4,  &&No\\

{\color{white} Nilradical }& [e_1, e_4] = e_6, [e_2, e_5] = e_6&a_{5,8}(e^{1,2}-e^{5,7}-e^{5,8})-a_{1,7}e^{1,7}&\textcolor[rgb]{1,1,1}{Maximal}\\

&[e_7,e_1]=e_1,[e_7,e_3]=-e_3,&														+a_{4,8}(\frac{1}{2}e^{1,3}-e^{4,8})+a_{3,8}(1/2.e^{3,7}-e^{3,8})&\\
&[e_7,e_5]=e_5,[e_7,e_6]=e_6,&														+\frac{1}{2}a_{6,8}(e^{1,4}+e^{2,5}-e^{6,7}-2e^{6,8})&\\
&[e_8,e_2]=e_2,[e_8,e_3]=2e_3,&														-a_{2,8}e^{2,8}-a_{7,8}e^{7,8}&\\
&[e_8,e_4]=2e_4,[e_8,e_5]=e_5,&Conditions: a_{6,8} \neq 0 ,  a_{3,8} a_{6,8}\neq a_{4,8}^{2}&{\color{white} Maximal }\\\
&[e_8,e_6]=2e_6&&\\
\hline

%%%%%%%%%%%%%%%%%%%%%%%%%%%%%%%%%%%%%%%%%%%%%%%%%%%%%%%%%%%%%%%%%%%%%%%%%%%%%%%%%%%%%%%%%%%%%%%%%%%%%%%%
\N_{6,14}&[e_1,e_3]=e_4,[e_1,e_4]=e_6, &\mbox{Dimension not even}&Yes\\

& [e_2, e_3] = e_5, [e_2, e_5] = a*e_6,&&\\

&[e_7,e_1]=e_1,[e_7,e_2]=e_2,&&\\
&[e_7,e_4]=e_4,[e_7,e_5]=e_5,&&\\
&[e_7,e_6]=2e_6,[e_8,e_1]=-\frac{1}{a}e_2,&&\\
&[e_8,e_2]=e_1,[e_8,e_4]=-\frac{1}{a}e_5,&&\\
&[e_8,e_5]=e_4,[e_9,e_3]=e_3,&&\\
&[e_9,e_4]=e_4,[e_9,e_5]=e_5,&&\\
&[e_9,e_6]=e_6&&\\
\hline

%%%%%%%%%%%%%%%%%%%%%%%%%%%%%%%%%%%%%%%%%%%%%%%%%%%%%%%%%%%%%%%%%%%%%%%%%%%%%%%%%%%%%%%%%%%%%%%%%%%%%%%%
\N_{6,15}& [e_1,e_2]=e_3+e_5,[e_1,e_3]=e_4, &\mbox{Dimension not even}&No\\

&[e_1, e_4] = e_6, [e_2, e_5] = e_6 &&\\

&[e_7,e_1]=e_1,[e_7,e_2]=2e_2,&&\\
&[e_7,e_3]=3e_3,[e_7,e_4]=4e_4,&&\\
&[e_7,e_5]=3e_5,[e_7,e_6]=5e_6,&&\\\hline

%%%%%%%%%%%%%%%%%%%%%%%%%%%%%%%%%%%%%%%%%%%%%%%%%%%%%%%%%%%%%%%%%%%%%%%%%%%%%%%%%%%%%%%%%%%%%%%%%%%%%%%%

\N_{6,16}&[e_1,e_3]=e_4,[e_1,e_4]=e_5,  &a_{4,8}(e^{1,3}-e^{4,7}-e^{4,8})-a_{1,7}e^{1,7}&No\\

& [e_1,e_5]=e_6,[e_2,e_3]=e_5, &					+a_{5,8}(e^{1,4}+e^{2,3}-2e^{5,7}-e^{5,8})&\\

&[e_2, e_4] = e_6,[e_7,e_1]=e_1&															+a_{6,8}(e^{1,5}+e^{2,4}-3e^{6,7}-e^{6,8})&\\
&[e_7,e_2]=2e_2,[e_7,e_4]=e_4,&																-a_{2,7}e^{2,7}-a_{3,8}e^{3,8}-a_{7,8}e^{7,8}&\\
&[e_7,e_5]=2e_5,[e_7,e_6]=3e_6,&																Conditions: &\\
&[e_8,e_3]=e_3,[e_8,e_4]=e_4,&a_{6,8}\neq 0,&\\
&[e_8,e_5]=e_5,[e_8,e_6]=e_6,&3 a_{3,8} a_{6,8}^{2}\neq 3 a_{4,8} a_{5,8} a_{6,8}+a_{5,8}^{3}&\\
\hline

%%%%%%%%%%%%%%%%%%%%%%%%%%%%%%%%%%%%%%%%%%%%%%%%%%%%%%%%%%%%%%%%%%%%%%%%%%%%%%%%%%%%%%%%%%%%%%%%%%%%%%%%
\N_{6,17}& [e_1,e_2]=e_3,[e_1,e_3]=e_4, &a_{3,8}(e^{1,2}-e^{7,3}-e^{3,8})-a_{1,7}e^{1,7}&No\\

&[e_1, e_4] = e_6, [e_2, e_5] = e_6,&			+a_{6,8}(e^{1,4}+e^{2,5}-3e^{6,7}-e^{6,8})&\\

&[e_7,e_1]=e_1,[e_7,e_3]=e_3,&																		+a_{4,8}(e^{1,3}-2e^{4,7}-e^{4,8})-a_{2,8}e^{2,8}&\\
&[e_7,e_4]=2e_4,[e_7,e_5]=3e_5,&-a_{5,7}e^{5,7}-a_{7,8}e^{7,8}&\\
&[e_7,e_6]=3e_6,[e_8,e_2]=e_2,&																Conditions:a_{6,8}\neq 0,& \\
&[e_8,e_3]=e_3,[e_8,e_4]=e_4,&2 a_{3,8} a_{6,8}\neq a_{4,8}&\\
&[e_8,e_6]=e_6&&\\
\hline

%%%%%%%%%%%%%%%%%%%%%%%%%%%%%%%%%%%%%%%%%%%%%%%%%%%%%%%%%%%%%%%%%%%%%%%%%%%%%%%%%%%%%%%%%%%%%%%%%%%%%%%%
\N_{6,18}& [e_1,e_2]=e_3,[e_1,e_3]=e_4, &\mbox{Never admits a symplectic}&Yes\\

&[e_1,e_4]=e_6,[e_2,e_3]=e_5,&\mbox{structure}&\\

&[e_2, e_5] = ae_6&&\\
&[e_7,e_1]=e_1,[e_7,e_2]=e_2,&&\\
&[e_7,e_3]=2e_3,[e_7,e_4]=3e_4,&&\\
&[e_7,e_5]=3e_5,[e_7,e_6]=4e_6,&&\\
&[e_8,e_1]=-\frac{1}{a}e_2,[e_8,e_2]=e_1,&&\\
&[e_8,e_4]=-\frac{1}{a}e_5,[e_8,e_5]=e_4,&&\\\hline

%%%%%%%%%%%%%%%%%%%%%%%%%%%%%%%%%%%%%%%%%%%%%%%%%%%%%%%%%%%%%%%%%%%%%%%%%%%%%%%%%%%%%%%%%%%%%%%%%%%%%%%%

\end{array}
\]
\[
\begin{array}{|c|l|c|c|} \hline 

\N_{6,19}& [e_1,e_2]=e_3,[e_1,e_3]=e_4, &\mbox{Dimension not even}&No\\

& [e_1,e_4]=e_5,[e_1,e_5]=e_6,&&\\
&[e_2, e_3] = e_6,&&\\
&[e_7,e_1]=e_1,[e_7,e_2]=3e_2,&&\\
&[e_7,e_3]=4e_3,[e_7,e_4]=5e_4,&&\\
&[e_7,e_5]=6e_5,[e_7,e_6]=7e_6,&&\\
\hline

%%%%%%%%%%%%%%%%%%%%%%%%%%%%%%%%%%%%%%%%%%%%%%%%%%%%%%%%%%%%%%%%%%%%%%%%%%%%%%%%%%%%%%%%%%%%%%%%%%%%%%%%
\N_{6,20}& [e_1,e_2]=e_3,[e_1,e_3]=e_4, &\mbox{Dimension not even}&No\\

&[e_1,e_4]=e_5,[e_1,e_5]=e_6, &&\\

&[e_2, e_3] = e_5, [e_2, e_4] = e_6&&\\
&[e_7,e_1]=e_1,[e_7,e_2]=2e_2,&&\\
&[e_7,e_3]=3e_3,[e_7,e_4]=4e_4,&&\\
&[e_7,e_5]=5e_5,[e_7,e_6]=6e_6&&\\
\hline

%%%%%%%%%%%%%%%%%%%%%%%%%%%%%%%%%%%%%%%%%%%%%%%%%%%%%%%%%%%%%%%%%%%%%%%%%%%%%%%%%%%%%%%%%%%%%%%%%%%%%%%%
\N_{6,21}&[e_1,e_2]=e_3,[e_1,e_5]=e_6,  &a_{3,8}(e^{1,2}-e^{3,7}-e^{3,8})-a_{1,7}e^{1,7}&No\\

& [e_2, e_3] = e_4, [e_2, e_4] = e_5,&\frac{1}{3}a_{6,8}(e^{1,5}+e^{3,4}-2e^{6,7}-3e^{6,8})&\\

&[e_3, e_4] = e_6,&\frac{1}{2}a_{4,8}(e^{2,3}-e^{4,7}-e^{4,8})-a_{2,8}e^{2,8}&\\
&[e_7,e_1]=e_1,[e_7,e_3]=e_3,&\frac{1}{3}a_{5,8}(e^{2,4}-e^{5,7}-e^{5,8})-a_{7,8}e^{7,8}&\\
&[e_7,e_4]=e_4,[e_7,e_5]=e_5,&Conditions:  a_{2,8} a_{6,8}\neq 8 a_{3,8} a_{5,8}+3 a_{4,8}^{2},&\\\
&[e_7,e_6]=2e_6,[e_8,e_2]=e_2,&a_{6,8}\neq 0&\\
&[e_8,e_3]=e_3,[e_8,e_4]=2e_4,&&\\
&[e_8,e_5]=3e_5,[e_8,e_6]=3e_6,&&\\
\hline

%%%%%%%%%%%%%%%%%%%%%%%%%%%%%%%%%%%%%%%%%%%%%%%%%%%%%%%%%%%%%%%%%%%%%%%%%%%%%%%%%%%%%%%%%%%%%%%%%%%%%%%%
\N_{6,22}& [e_1, e_2] = e_3, [e_1, e_3] = e_5, &\mbox{Dimension not even}&No\\

&[e_1,e_5]=e_6,[e_2,e_3]=e_4, &&\textcolor[rgb]{1,1,1}{Maximal}\\

&[e_2, e_4] = e_5, [e_3, e_4] = e_6&&\\
&[e_7,e_1]=e_1,[e_7,e_2]=\frac{1}{2}e_2,&&\\
&[e_7,e_3]=\frac{3}{2}e_3,[e_7,e_4]=2e_4,&&\\
&[e_7,e_5]=\frac{5}{2}e_5,[e_7,e_6]=\frac{7}{2}e_6&&\\
\hline

\end{array}
\]

\end{document}